\documentclass[11pt]{amsart}
\usepackage{amsmath,amsthm,amssymb,amssymb, paralist, xspace, graphicx, url, amscd, euscript, mathrsfs,stmaryrd,epic,eepic,color}
\usepackage[all]{xy}
\usepackage{amsmath,amscd}
\SelectTips{cm}{}

\usepackage[colorlinks=true]{hyperref}

\usepackage{pstricks}

\numberwithin{equation}{subsection}

1

\numberwithin{subsection}{section}

\allowdisplaybreaks[1]

\newtheorem*{namedtheorem}{\theoremname}
\newcommand{\theoremname}{testing}

\theoremstyle{plain}

\newtheorem{thm}{Theorem}[section]
\newtheorem{proposition}[thm]{Proposition}
\newtheorem{proposition-definition}[thm]{Proposition-Definition}
\newtheorem{lemma-definition}[thm]{Lemma-Definition}

\newtheorem{lemma}[thm]{Lemma}

\theoremstyle{definition}

\newtheorem{notation}[thm]{Notation}

\newtheorem{remark}[thm]{Remark}

\newtheorem{construction}[thm]{Construction}

\theoremstyle{remark}

\numberwithin{thm}{section}

\def\P{\mathbb{P}}
\def\A{\mathbb{A}}

\def\C{\mathbb{C}}

\newcommand\RR{\mathbb{R}}

\newcommand\arr{\ifinner\to\else\longrightarrow\fi}

\def\displaytimes_#1{\mathrel{\mathop{\times}\limits_{#1}}}

\def\displayotimes_#1{\mathrel{\mathop{\bigotimes}\limits_{#1}}}

\newcommand\Aut{\operatorname{Aut}}

\newcommand\rank{\operatorname{rank}}

\newdir{ >}{{}*!/-5pt/@{>}}

\newcommand\doublelong[2]{\mathbin{\xymatrix{{}\ar@<3pt>[r]^{#1}
\ar@<-3pt>[r]_{#2}&}}}

\newlength{\ignora}

\newcommand{\Z}{\mathbb{Z}}

\numberwithin{equation}{subsection}

\begin{document}

%%%%%%%%%%%%%%%%%%%%%%%%%%%%%%%%%%%%%%%%%%%%%%%%

\title{Log rationally connected surfaces}

%\author{Qile Chen}

\author{Yi Zhu}

\address[Zhu]{Department of Mathematics\\
University of Utah\\
Room 233\\
155 S 1400 E \\
Salt Lake City, UT 84112\\
U.S.A.}
\email{yzhu@math.utah.edu}

%\thanks{Chen is partially supported by NSF grant DMS-1403271.}

%\subjclass[2010]{??}
%\keywords{???}

%\date{\today}
\begin{abstract}
In this paper, combining the works of Miyanishi-Tsunoda \cite{Miyanishi-T2,Miyanishi-T1} and Keel-McKernan \cite{KM}, we prove the log Castelnuovo's rationality criterion for smooth quasiprojective surfaces over complex numbers. 
%a version of Koll\'ar's theorem over large fields in the logarithmic setting. 

\end{abstract}
\maketitle

%\tableofcontents

%%%%%%%%%%%%%%%%%%%%%%%%%%%%%%%%%%%%%%%%%%%%%%%%Introduction
\section{Introduction}\label{sec:intro}

Iitaka \cite{Iitaka77} proposes a program generalizing the classical theory of complex projective surfaces to open algebraic surfaces using the language of pairs. The classification theory of open algebraic surfaces according to the log Kodaira dimension has been carried out by works of Iitaka \cite{Iitaka81, Iitaka-AG}, Kawamata \cite{Kawamata79}, Miyanishi-Tsunoda \cite{Miyanishi-T2,Miyanishi-T1} and Keel-McKernan \cite{KM}. %Iitaka's philosophy claims that whenever there is a theorem for compact projective varieties, there should be a counter-theorem for smooth quasiprojective varieties after the extension to log pairs.

%The classification theory of open algebraic surfaces 

Open algebraic surfaces with log Kodaira dimension $-\infty$ are particularly interesting. Basic examples are \emph{log ruled (log uniruled)} log surface pairs, i.e., the interior contains a Zariski open subset isomorphic to (dominated by) $\A^1\times T_0$, where $\dim T_0=1$. Keel and McKernan \cite{KM} prove the following log Enriques' ruledness criterion.

\begin{thm}[\cite{KM}]
	A log smooth log surface pair $(X,D)$ is log uniruled if and only if $\kappa(X,D)=-\infty$.
\end{thm}

A \emph{log rational curve} on a log pair $(X,D)$ is a rational curve $f:\P^1\to X$ which meets $D$ at most once. Inspired by the theory of log rational curves developed by Chen and the author \cite{A1,CZ}, we define that a log pair $(X,D)$ is \emph{log rationally connected} if there exists a log rational curve passing through a general pair of points.

%a general pair of points on $U$ can be connected by the image of an $\A^1$-curve on $U$, i.e., a morphism from $\A^1$ to $U$.

%For log smooth pairs, this definition is equivalent that

In this paper, combining the works of Miyanishi-Tsunoda \cite{Miyanishi-T2},\linebreak \cite{Miyanishi-T1} and Keel-McKernan \cite{KM}, we obtain the following numerical criterion for log rationally connected log surface pairs, which generalizes the classical Castelnuovo's rationality criterion.

\begin{thm}[Log Castelnuovo's criterion]\label{thm:main} The following statements are equivalent for a log smooth log surface pair $(X,D)$:\begin{enumerate}
		\item $(X,D)$ is log rationally connected;
		\item $H^0(X,(\Omega^1_X(\log D))^{\otimes m})=0$, for any $m\ge 1$;
		\item $\kappa(X,D)=-\infty$ and $H^0(X,S^{12}\Omega^1_X(\log D))=0$.
	\end{enumerate}   
\end{thm} 

%Since log rational curves do not admit pluri log one forms, $(1)$ implies $(2)$.  
%$(1)\Rightarrow(2)\Rightarrow (3)$ is trivial because log rational curves do not admit pluri log one forms. 
By \cite[Prop. 2.7]{A1}, log rational connectedness implies there exists a log rational curve $f:\P^1\to X$ such that $f^*\Omega^1_X(\log D)$ is the sum of line bundles of negative degree. Hence every section of $(\Omega^1_X(\log D))^{\otimes m}$ vanishes along $f(\P^1)$. Since such curves cover a dense open subset of $X$, $(1)$ implies $(2)$. It is obvious that $(2)$ implies $(3)$. For $(3)\Rightarrow (1)$, the idea is to run the log minimal model program for log surface pairs. The end products are classified by Miyanishi-Tsunoda \cite{Miyanishi-T2,Miyanishi-T1} and Keel-McKernan \cite{KM}. And many of them are log rationally connected with the exception of the log-ruled case. The log ruled case is proved in Proposition \ref{prop:main}. In the proof, we discover, somewhat unexpectedly, a close relationship between the non-existence of pluri log one forms and the solutions for the strong approximation problem of the affine line over function fields of complex curves.

%The equivalence of the first two statements in Theorem \ref{thm:main} is . 
\begin{remark}{\ }\begin{enumerate} 
		%\item 
		%
		\item Log rationally connected log pairs are expected to be the natural geometric context for the strong approximation problem over function fields of curves, c.f., \cite{HT-log-Fano,rankone,SA}. 
		
		\item By \cite[p.91]{Iitaka81}, the precise copy of Castelnuovo's rationality criterion fails. There exists a log surface pair $(X,D)$ with $q(X,D)=0$, $h^0(X,\mathcal{O}_X(2(K_X+D)))=0$ but $h^0(X,\mathcal{O}_X(4(K_X+D)))=1$. In fact, a more careful analysis of our proof will show that the condition (3) in Theorem \ref{thm:main} is sharp. 
		%\item Condition (3) is sharp. See Section \ref{sec:eg} for more counterexamples
		
		\item By \cite[Cor. 7.9]{KM}, log rational connectedness implies that the fundamental group of the interior is finite. However, this topological restriction plus negative log Kodaira dimension will not characterize log rational connectedness. See Propositon \ref{prop:e}.
		
		\item The log smooth log Fano surface pair $(\P^2,\{xy=0\})$ has nonvanishing log irregularity and infinite fundamental group of the interior. In particular, it is not log rationally connected. So the log analogue of the theorem of Campana \cite{Campana} and Koll\'ar-Miyaoka-Mori \cite{KMM} fails even for surface pairs.
	\end{enumerate}
	
\end{remark}

%The following question is work in progress.

%\begin{question}[Chenyang Xu]Is every log rationally connected surface pair strongly log rationally connected, i.e., every pair of points can be connected by a very free log rational curve?\end{question}

%\begin{thm}[Castelnuovo's rationality theorem]Let $X$ be a smooth projective surface. Then the following statements are equivalent:\begin{enumerate}\item $X$ is rational;\item $X$ is rationally connected;\item $q(X)=0$ and $\kappa(X)=-\infty$;\item $q(X)=p_2(X)=0$.\end{enumerate}\end{thm}

%In this paper, we prove the following logarithmic version of Castelnuovo's rationality theorem.

\begin{notation}
	In this paper, we work with varieties and log pairs over complex numbers $\C$. A log pair $(X,D)$ means a variety $X$ with a reduced Weil divisor $D$. Let $U$ be its interior $X-D$. We say that $(X,D)$ is \emph{log smooth} if $X$ is smooth and $D$ is a simple normal crossing divisor. A log pair is projective if the ambient variety is projective.
	
	For a log smooth pair $(X,D)$, we use $\kappa(X,D)$ to denote the logarithmic Kodaira dimension and $q(X,D)$ to denote the logarithmic irregularity, i.e., $q(X,D)=h^0(X,\Omega_X^1(\log D))$. They only depend on the interior of the pair.

\end{notation}

\section{Log minimal model program}
Let $(X,D)$ be a projective log smooth surface pair with $\kappa(X,D)=-\infty$. By \cite[Theorem 3.47]{Kollar-Mori}, we run the log minimal model program on this pair
$$(X,D)=(X_0,D_0)\to(X_1,D_1)\to\cdots\to (X_k,D_k)=(X^*,D^*),$$
such that: 
\begin{enumerate}
	\item each step is a divisorial contraction;
	\item the log Kodaira dimension remains the same, i.e., $\kappa(X_i,D_i)=-\infty$;
	\item the end product $(X^*,D^*)$ is either \begin{enumerate}
		\item log ruled, or 
		\item a log del Pezzo surface of Picard number one, i.e., $\rho(X^*)=1$.
	\end{enumerate}
\end{enumerate}

\begin{lemma}\label{lem:logRC}
	If $(X^*,D^*)$ is log ruled (log rationally connected), so is $(X,D)$.
\end{lemma}

\begin{proof} At each step of the log MMP, we contract a curve representing a $(K_X+D)$-negative extremal ray. Such curve could be contracted to either a boundary point or an interior point. Thus after deleting $U:=X-D$ with a pure codimension one subset, it is isomorphic to a dense open $V^*$ of $U^*:=X^*-D^*$ such that $\dim(U^*-V^*)=0$. The lemma follows from log ruledness (log rational connectedness) remains after deleting finitely many points.
\end{proof}

When $(X,D)$ is a log del Pezzo surface pair, we recall the following two fundamental theorems established by Miyanishi-Tsunoda \cite{Miyanishi-T2,Miyanishi-T1} and Keel-McKernan \cite{KM}.

\begin{thm}\cite[Ch.2, Thm 5.1.2]{Miyanishi}\label{thm:MT}
	Let $(X,D\neq\emptyset)$ be a log del Pezzo surface pair of Picard number one with nonempty boundary. If $(X,D)$ is not log ruled, then the interior $U:=X-D$ is a Platonic $\A^1_*$-fiber space, i.e., $U$ is isomorphic to $\A^2-(0)$ modulo a noncyclic small finite subgroup of $GL(2,\C)$. In particular, $(X,D)$ is log rationally connected.
\end{thm}

\begin{thm}\cite[Theorem 1.6]{KM}\label{thm:KM}
	Let $(X,D=\emptyset)$ be a log del Pezzo surface pair with empty boundary. Then the smooth locus of $X^{sm}$ is rationally connected.  
\end{thm}

\begin{proof}[Proof of Theorem \ref*{thm:main}]
	If $(X^*,D^*)$ a log del Pezzo surface pair of Picard number one and not log ruled, the theorem follows from Theorem \ref{thm:MT}, \ref{thm:KM} and Lemma \ref{lem:logRC}. If $(X^*,D^*)$ is log ruled, then by Lemma \ref{lem:logRC} again, $(X,D)$ is log ruled. This case will be proved in Proposition \ref{prop:main}.
\end{proof}

\section{Log ruled case}

%\begin{lemma}\cite[Prop. 11.3 and 11.4]{Iitaka-AG}\label{lem:proper-bir}
%Both the logarithmic Kodaira dimension and logarithmic irregularity are proper birational invariants.\qed
%\end{lemma}

\begin{lemma}\label{lem:extend-T}
	Let $(X,D)$ a projective log smooth surface pair. If $(X,D)$ is log ruled and let $f:X\dashrightarrow T$ be a rational map to a smooth projective curve $T$, then there exists a birational morphism $(X',D')\to (X,D)$ extending $f$ and $f|_{X'-D'}$ is proper onto the image $X-D$.
\end{lemma}

\proof To resolve the indeterminacy, we need to take a sequence of blow ups. For each intermediate blow up $g_i:X_{i+1}\to (X_i,D_i)$, we take $D_{i+1}:=g_i^{-1}(D_i)$. Then it is clear that $g_i|_{X_{i+1}-D_{i+1}}$ is proper onto the image $X_i-D_i$.\qed 

\begin{proposition}\label{prop:main}
	If $(X,D)$ is log ruled and $H^0(X,S^{12}\Omega^1_X(\log D))=0$, then $(X,D)$ is log rationally connected.
\end{proposition}

\proof By \cite[Prop. 11.3]{Iitaka-AG} and Lemma \ref{lem:extend-T}, we may assume that the rational map $(X,D)\to T_0$ extends to a proper flat morphism $f:(X,D)\to T$, where $T$ is a smooth projective curve containing $T_0$. The log ruling on the generic fiber implies that $f(U)$ is nonempty open in $T$. Let $S$ be the complement of $f(U)$ in $T$. By \cite[2.1.17]{Miyanishi}, we have a morphism of pairs
$$f:(X,D)\to (T,S),$$
which induces an injection $$H^0(T, \Omega_T^1(\log S))\to H^0(X, \Omega_X^1(\log D)).$$
Thus, by assumption, ${q}(T,S)=0$. The classification of log curves implies that the pair $(T,S)$ is isomorphic to either $(\P^1,\emptyset)$ or $(\P^1,\{\infty\})$. 

\begin{notation}
	Let $p_1,\cdots,p_k$ be the points on $T-S$ whose inverse image $f|_U^{-1}(p_i)$ contains no reduced component. Let $d_i$ be the minimal multiplicity of all irreducible components of $f|_U^{-1}(p_i)$. We may assume that $d_1\le \cdots\le d_k$.
\end{notation}

\begin{lemma}\label{lem:A}
	If the base $(T,S)$ is isomorphic to $(\P^1,\{\infty\})$ and $$H^0(X,S^2\Omega^1_X(\log D))=0,$$ then $k\le 1$.
\end{lemma}

\proof If $k\ge 2$, we may assume that $p_1=0$ and $p_2=1$. Let $z$ be the coordinate on $\P^1-\{\infty\}$. Since any irreducible component $E$ of $f|_U^{-1}(p_i)$ has the form $(t^e=0)$ where $e\ge 2$, the pullback of the tensor 
$$f^*\left(\frac{dz}{z}\otimes \frac{dz}{z-1}\right)=e\frac{dt}{t}\otimes t^{e-1}dt=et^{e-2}dt\otimes dt$$
is regular on $E$.
Thus, we have $$0\neq f^*\left(\frac{dz}{z}\otimes \frac{dz}{z-1}\right)\in H^0(X,S^2\Omega^1_X(\log D)),$$ which contradicts the assumption.\qed 

\begin{lemma}\label{lem:P}
	If the base $(T,S)$ is isomorphic to $(\P^1,\emptyset)$ and $$H^0(X,S^{12}\Omega^1_X(\log D))=0,$$ then we have the following:
	\begin{enumerate}
		\item $k=0$, i.e., each fiber contains a reduced component;
		\item $k=1$ and $d_1=n\ge 2$;
		\item $k=2$ and $d_1,d_2\ge 2$;
		\item $k=3$ and $(d_1,d_2,d_3)$ are one of the following triples:
		\begin{enumerate}
			\item $(2,2,n\ge2)$;
			\item $(2,3,3)$;
			\item $(2,3,4)$;
			\item $(2,3,5)$.
		\end{enumerate}
	\end{enumerate}
\end{lemma}

\proof 
If $k\ge 4$, let $p_1=0, p_2=1, p_3=\infty, p_4=c$. By the same argument as in Lemma \ref{lem:A}, the pullback 
$$f^*\left(\frac{dz}{z}\otimes \frac{dz}{(z-1)(z-c)}\right)$$
gives a nonzero element in $H^0(X,S^2\Omega^1_X(\log D))$.

Now we may assume that $k=3$ and $p_1=0, p_2=1, p_3=\infty$. Consider the pullback of one forms and we obtain the following table.
\begin{center}
	\begin{tabular}{| l || c | c | c|}
		\hline
		order at & $0$& $1$& $\infty$\\ \hline \hline
		$f^*(\frac{dz}{z})$ & $-1$ & $d_2-1$ & $-1$\\ \hline
		$f^*(\frac{dz}{z-1})$ & $d_1-1$ & $-1$ &$-1$\\ \hline
		$f^*(\frac{dz}{z(z-1)})$ & $-1$ & $-1$ &$d_3-1$\\
		\hline
	\end{tabular}
\end{center}

With the same argument as in Lemma \ref{lem:A}, we obtain the following:\begin{itemize}
	\item If $d_1\ge 3$, then $$0\neq f^*\left(\frac{dz}{z}\otimes\frac{dz}{z-1}\otimes\frac{dz}{z(z-1)}\right)\in H^0(X,S^3\Omega^1_X(\log D)).$$
	\item If $d_1=2, d_2\ge 4$, then $$0\neq f^*\left(\frac{dz}{z}\otimes\left(\frac{dz}{z-1}\right)^{\otimes 2}\otimes\frac{dz}{z(z-1)}\right)\in H^0(X,S^4\Omega^1_X(\log D)).$$
	\item If $d_1=2, d_2=3, d_3\ge 6$, then 
	$$0\neq f^*\left(\left(\frac{dz}{z}\right)^{\otimes 3}\otimes(\frac{dz}{z-1})^{\otimes 2}\otimes\frac{dz}{z(z-1)}\right)\in H^0(X,S^6\Omega^1_X(\log D)).$$
	
\end{itemize}

All above cases contradict with $H^0(X,S^{12}\Omega^1_X(\log D))=0$. Therefore the lemma is proved. \qed

%$$f^*(\frac{dz}{z}\otimes\frac{dz}{z-1}\otimes\frac{dz}{z(z-1)})$$

\begin{lemma}
	With the same notations as above, assume that $$H^0(X,S^{12}\Omega^1_X(\log D))=0.$$
	Then there exists a finite cover $g:\P^1\to \P^1$ such that after taking the base change
	\begin{equation*}
	\xymatrix{
		(X',D') \ar[r]^{g'} \ar[d]^{f'} & (X,D)\ar[d]^f \\
		(T'=\P^1,S'=g^{-1}(S)) \ar[r]^-g & (T,S),
	}
	\end{equation*}
	the fiber $U'_t=X'_t-D'_t$ contains a reduced component for every $t\in T'-S'$. 
\end{lemma}

\begin{proof} 
	When $k\le 2$, we can choose $g$ as a cyclic cover $z\mapsto z^n$. Now by Lemma \ref{lem:A} and \ref{lem:P}, the only cases left are when the base is $(\P^1, \emptyset)$ and $k=3$. On the other hand, every such case corresponds to a finite noncyclic subgroup $$\Gamma \le \Aut(\P^1)=PGL(2,\C)\cong SO(3;\RR)$$ as in \cite[Chapter 5, Theorem 9.1]{Artin-alg}.  
	The natural quotient map $$g:\P^1\to\P^1/\Gamma\cong \P^1$$ totally ramifies over three points on the target, say, $p_1, p_2, p_3$, such that the ramification order of every preimage of $p_i$ is $d_i$, for $i=1,2,3$. The lemma follows by taking base change via the map $g:\P^1\to\P^1/\Gamma$. 
\end{proof}

Now let us return to the proof of Proposition \ref{prop:main}. Let $U'$ be the interior of $(X',D')$. By construction, $f'$ is a morphism of log pairs. And every fiber of $f'|_{U'}$ contains a reduced component, i.e., there exists local integral section over any point in $T-S$. 

Note that $f'$ gives an integral model of $\P^1$ or $\A^1$ over the function field $\C(\P^1)$. Since strong approximation holds for $\P^1$ or $\A^1$ over the function field of curves \cite[Theorem 6.13]{Rosen}, a general pair of points in $U'$ can be connected by an integral section, which is either a rational curve or an $\A^1$-curve. The log rational connectedness of $U$ follows because $U'$ maps surjectively to $U$.\qed

\section{An example}

\begin{proposition}\label{prop:e}
	There exists an log ruled surface $(X,D)$ such that 
	\begin{enumerate}
		\item $\pi_1(X-D)$ is trivial;
		\item $q(X,D)=0$;
		\item $\kappa(X,D)=-\infty$;
		\item $(X,D)$ is not log rationally connected.
	\end{enumerate}
\end{proposition}

\begin{construction}
	Take $U=\A^1\times \A^1$ with the natural projection $\pi$ to the first factor $T=\A^1$. Pick two points $p_1, p_2$ on $T$. For each $i=1,2$, remove $\pi^{-1}(p_i)$ from $U$ and glue back a disjoint union of a triple $\A^1$ and a double $\A^1$. We obtain a new surface $q:U'\to T$. We can achieve this by embedding $U$ into $\P^1\times \P^1$, taking further blow ups and deleting some extra divisors. Take any log smooth model $(X,D)$ with the interior $U'$.
\end{construction}

\proof[Proof of Proposition \ref{prop:e}] First $q:U'\to T$ gives an $\A^1$-ruling over $\A^1$, in particular, we have $\kappa(X,D)=-\infty$. The proof of Lemma \ref{lem:A} implies that $(X,D)$ is not log rationally connected. For the fundamental group of $U'$, an easy calculation via Van Kampen's theorem shows that it is simply connected. Finally by the construction of log Albanese variety \cite{Iitaka77} and the rationality of $X$, we have $$\rank H^0(\Omega_X^1(\log D))= \rank H_1(U',\Z)^{free}.$$
The later is trivial because $\pi_1(U')=0$. Therefore we have $q(X,D)=0$.\qed

%\begin{proposition}
%There exists an $\A^1$-ruled surface pair $(X,D)$ with $\pi_1(U)=0$
%\end{proposition}

%In the first case, we may replace $(X,D)$ by $(X, D+f^{-1}(\infty))$ because the log rational connectedness of the later implies the log rational connectedness of the former. Now we have a morphism  of log pairs$$f:(X,D)\to (\P^1,\{\infty\}),$$where the restriction $f|_U:U\to\A^1$ is surjective. Let $b_1,\cdots,b_k$ be the points on $\A^1$ where $f|_U^{-1}(b_i)$ is not reduced and let $m_i$ be the least common multiple of $c_{ij}$, where $f|_U^{-1}(b_i)=\sum_j c_{ij}F_{ij}$. Let $(X',D')$ be the base change:

%where $g$ is given by the polynomial $(x-b_1)^{m_1}\cdots(x-b_k)^{m_k}$ and $D':=g'^{-1}(D)$. 

\subsection*{Acknowledgments}
The author would like to thank Qile Chen, Tommaso de Fernex, J\'anos Koll\'ar, Jason Starr, Zhiyu Tian and Chenyang Xu for helpful discussions.

\def\MR#1{{\tt MR #1}}

%\address{Department of Mathematics, University of Utah\\
%	Room 233, 155 S 1400 E, Salt Lake City, UT 84112, USA\\
%	\email{yzhu@math.utah.edu}\\
%	\received{January 17, 2015}}


\begin{thebibliography}{KMM92}
	
	\bibitem[Art91]{Artin-alg}
	Michael Artin, \emph{Algebra}, Prentice Hall, Inc., Englewood Cliffs, NJ, 1991.
	\MR{1129886 (92g:00001)}
	
	\bibitem[Cam92]{Campana}
	F.~Campana, \emph{Connexit\'e rationnelle des vari\'et\'es de {F}ano}, Ann.
	Sci. \'Ecole Norm. Sup. (4) \textbf{25} (1992), no.~5, 539--545. \MR{1191735
		(93k:14050)}
	
	%\bibitem[CZ14a]{rankone}
	%Qile Chen and Yi~Zhu, \emph{$\mathbb{A}^1$-connected varieties of rank one over  nonclosed fields}, {\tt arXiv:1409.6398}.
	
	\bibitem[CZ14a]{A1}
	Qile Chen and Yi~Zhu, \emph{$\mathbb{A}^1$-curves on log smooth varieties}, {\tt arXiv:1407.5476},
	(2014).
	
	\bibitem[CZ14b]{CZ}
	Qile Chen and Yi~Zhu, \emph{Very free curves on {F}ano complete intersections}, Algebr.
	Geom. \textbf{1} (2014), no.~5, 558--572. \MR{3296805}
	
	%\bibitem[CZ14c]{CZ}
	%Qile Chen and Yi~Zhu, \emph{Very free curves on fano complete intersections}, Algebraic
	% Geometry \textbf{1} (2014), no.~5, 558--572, {\tt arXiv:1311.7189}.
	
	\bibitem[CZ15]{SA}
	Qile Chen and Yi~Zhu, \emph{Strong approximation over function fields},
	{\tt arXiv:1510.04647}, (2015).
	
	\bibitem[CZ16]{rankone}
	Qile Chen and Yi~Zhu, \emph{{$\Bbb A^1$}-connected varieties of rank one over
		nonclosed fields}, Math. Ann. \textbf{364} (2016), no.~3-4, 1505--1515.
	\MR{3466876}
	
	\bibitem[HT08]{HT-log-Fano}
	Brendan Hassett and Yuri Tschinkel, \emph{Log {F}ano varieties over function
		fields of curves}, Invent. Math. \textbf{173} (2008), no.~1, 7--21.
	\MR{2403393 (2009c:14080)}
	
	\enlargethispage{1em}
	\bibitem[Iit77]{Iitaka77}
	S.~Iitaka, \emph{On logarithmic {K}odaira dimension of algebraic varieties},
	Complex analysis and algebraic geometry, Iwanami Shoten, Tokyo, 1977,
	pp.~175--189. \MR{0569688 (58 \#27975)}
	
	\bibitem[Iit81]{Iitaka81}
	Shigeru Iitaka, \emph{Birational geometry for open varieties}, S\'eminaire de
	Math\'ematiques Sup\'erieures [Seminar on Higher Mathematics], vol.~76,
	Presses de l'Universit\'e de Montr\'eal, Montreal, Que., 1981. \MR{647148
		(83j:14011)}
	
	\bibitem[Iit82]{Iitaka-AG}
	Shigeru Iitaka, \emph{Algebraic geometry}, Graduate Texts in Mathematics, vol.~76,
	Springer-Verlag, New York-Berlin, 1982, An introduction to birational
	geometry of algebraic varieties, North-Holland Mathematical Library, 24.
	\MR{637060 (84j:14001)}
	
	\bibitem[Kaw79]{Kawamata79}
	Yujiro Kawamata, \emph{On the classification of noncomplete algebraic
		surfaces}, Algebraic geometry ({P}roc. {S}ummer {M}eeting, {U}niv.
	{C}openhagen, {C}openhagen, 1978), Lecture Notes in Math., vol. 732,
	Springer, Berlin, 1979, pp.~215--232. \MR{555700 (81c:14021)}
	
	\bibitem[KM98]{Kollar-Mori}
	J{\'a}nos Koll{\'a}r and Shigefumi Mori, \emph{Birational geometry of algebraic
		varieties}, Cambridge Tracts in Mathematics, vol. 134, Cambridge University
	Press, Cambridge, 1998, With the collaboration of C. H. Clemens and A. Corti,
	Translated from the 1998 Japanese original. \MR{1658959 (2000b:14018)}
	
	\bibitem[KM99]{KM}
	Se{\'a}n Keel and James McKernan, \emph{Rational curves on quasi-projective
		surfaces}, Mem. Amer. Math. Soc. \textbf{140} (1999), no.~669, viii+153.
	\MR{1610249 (99m:14068)}
	
	\bibitem[KMM92]{KMM}
	J{\'a}nos Koll{\'a}r, Yoichi Miyaoka, and Shigefumi Mori, \emph{Rationally
		connected varieties}, J. Algebraic Geom. \textbf{1} (1992), no.~3, 429--448.
	\MR{1158625 (93i:14014)}
	
	\bibitem[Miy01]{Miyanishi}
	Masayoshi Miyanishi, \emph{Open algebraic surfaces}, CRM Monograph Series,
	vol.~12, American Mathematical Society, Providence, RI, 2001. \MR{1800276
		(2002e:14101)}
	
	\bibitem[MT84a]{Miyanishi-T2}
	Masayoshi Miyanishi and Shuichiro Tsunoda, \emph{Logarithmic del {P}ezzo
		surfaces of rank one with noncontractible boundaries}, Japan. J. Math. (N.S.)
	\textbf{10} (1984), no.~2, 271--319. \MR{884422 (88b:14030)}
	
	\bibitem[MT84b]{Miyanishi-T1}
	Masayoshi Miyanishi and Shuichiro Tsunoda, \emph{Noncomplete algebraic surfaces with logarithmic {K}odaira
		dimension {$-\infty$} and with nonconnected boundaries at infinity}, Japan.
	J. Math. (N.S.) \textbf{10} (1984), no.~2, 195--242. \MR{884420 (88b:14029)}
	
	\bibitem[Ros02]{Rosen}
	Michael Rosen, \emph{Number theory in function fields}, Graduate Texts in
	Mathematics, vol. 210, Springer-Verlag, New York, 2002. \MR{1876657
		(2003d:11171)}
	
\end{thebibliography}
\end{document}